\documentclass[10pt]{article}
\usepackage{amssymb,bbm,amsthm}
\usepackage{amsmath}
\usepackage{amsbsy}
\usepackage{cite}
\usepackage[dvips]{graphicx}

\newcommand{\as}{\\[.6em]}

\newcommand{\AS}{\\[1.2em]}
\newcommand{\dis}{\displaystyle}
\newcommand{\bela}[1]{\begin{equation}\label{#1}}
\newcommand{\ela}{\end{equation}}
\newcommand{\bear}[1]{\begin{array}{#1}}
\newcommand{\ear}{\end{array}}

\newcommand{\br}{\mbox{\boldmath $r$}}

\newcommand{\ba}{\mbox{\boldmath $a$}}
\newcommand{\bb}{\mbox{\boldmath $b$}}

\newcommand{\bM}{\mbox{\boldmath $M$}}

\newcommand{\diag}{\operatorname{diag}}

\newcommand{\Z}{\mathbbm{Z}}
\renewcommand{\P}{\mathbbm{P}}
\newcommand{\R}{\mathbbm{R}}
\newcommand{\C}{\mathbbm{C}}

\newcommand{\n}{\mbox{\boldmath $n$}}

\newcommand{\Msf}{\mathsf{M}}
\newcommand{\Vsf}{\mathsf{V}}
\newcommand{\Wsf}{\mathsf{W}}
\newcommand{\lsf}{\mathsf{l}}
\newcommand{\psf}{\mathsf{p}}
\newcommand{\asf}{\mathsf{a}}
\newcommand{\bsf}{\mathsf{b}}
\newcommand{\xsf}{\mathsf{x}}
\newcommand{\ysf}{\mathsf{y}}
\newcommand{\Bsf}{\mathsf{B}}

\newcommand{\calM}{\mathcal{M}}

\theoremstyle{plain}
\newtheorem{theorem}{Theorem}[section]

\newtheorem{lemma}[theorem]{Lemma}

\theoremstyle{definition}
\newtheorem{definition}[theorem]{Definition}
\newtheorem{remark}[theorem]{Remark}

\numberwithin{equation}{section}

\begin{document}

\begin{center} 
 {\Large\bf Discrete line complexes and integrable\vspace{2mm} evolution of minors}\,\footnote{This research was supported by the DFG Collaborative Research Centre SFB/TRR 109 {\em Discretization in Geometry and Dynamics} and the Australian Research Council (ARC).}
\end{center}

\smallskip

\begin{center}
 \sc
 A.I.\ Bobenko$\,^{2}$ and W.K.\ Schief$\,^{3,4}$
\end{center}

\smallskip

\begin{center}
 \small\sl
$^2$ Institut f\"ur Mathematik,Technische Universit\"at Berlin, Stra\ss e des 17.\ Juni 136, 10623 Berlin, Germany\\[2mm]
$^3$ School of Mathematics and Statistics, The University of New South Wales, Sydney, NSW 2052, Australia\\[2mm]
$^4$ Australian Research Council Centre of Excellence for Mathematics and Statistics of Complex Systems, School of Mathematics and Statistics, The University of New South Wales, Sydney, NSW 2052, Australia
\end{center}

\begin{abstract}
Based on the classical Pl\"ucker correspondence, we present algebraic and geometric properties of discrete integrable line complexes in $\C\P^3$. Algebraically, these are encoded in a discrete integrable system which appears in various guises in the theory of continuous and discrete integrable systems. Geometrically, the existence of these integrable line complexes is shown to be guaranteed by Desargues' classical theorem of projective geometry. A remarkable characterisation in terms of correlations of $\C\P^3$ is also recorded.
\end{abstract}

\section{Introduction}
Line congruences, that is, two-parameter families of lines, constitute fundamental objects in classical differential geometry \cite{Finikov59}. In particular, normal and Weingarten congruences have been studied in great detail \cite{Eisenhart60}. Their importance in connection with the geometric theory of integrable systems has been well documented (see \cite{Doliwa01} and references therein). Recently, in the context of integrable discrete differential geometry \cite{BobenkoSuris09}, attention has been drawn to {\em discrete line congruences}~\cite{DoliwaSantiniManas00}, that is, two-parameter families of lines which are (combinatorially) attached to the vertices of a $\Z^2$ lattice. Discrete Weingarten congruences have been shown to lie at the heart of the B\"acklund transformation for discrete pseudospherical surfaces \cite{Sauer50,Wunderlich51,BobenkoPinkall96}.  Discrete normal congruences have been used to define Gaussian and mean curvatures and the associated Steiner formula for discrete analogues of surfaces parametrised in terms of curvature coordinates \cite{Schief03,Schief06,BobenkoPottmannWallner2010}. Discrete line congruences have also found important applications in architectural geometry \cite{WangJiangBompasPottmann13}.

Here, we are concerned with the extension of discrete line congruences in a three-dimensional complex projective space $\C\P^3$ to three-dimensional ``lattices of lines'', that is, maps of the form
\bela{I1}
  \lsf : \Z^3\rightarrow{\{\mbox{lines in $\C\P^3$}}\}.
\ela
The latter define three-parameter families of lines which may be termed {\em discrete line complexes}. In the following, attention could be restricted to discrete line complexes in a real projective space $\R\P^3$. However, in order to analyse integrable line complexes in classical (Lie and Pl\"ucker) sub-geometries of projective geometry, it is required to consider a complex ambient space. The corresponding rich geometric structure is the subject of a forthcoming publication. Discrete line complexes in a four-dimensional (projective) space have been shown to be integrable if any two neighbouring lines intersect \cite{DoliwaSantiniManas00}. Here, integrability is understood in the sense of multi-dimensional consistency, that is, in a well-defined sense, these discrete line complexes may be extended to $\Z^N$ lattices of lines such that the restriction to any $\Z^3$ sub-lattice constitutes a discrete line complex.

This paper is based on the observation that a system of algebraic equations which arises in various avatars in the theory of both continuous and discrete integrable systems may be interpreted without reference to its origins as a system of integrable discrete equations. Thus, in Sections 2 and 3, we set down this privileged discrete integrable system ({\em $M$-system}) and indicate how it is related to, for instance, the important Fundamental and Darboux transformations of classical differential geometry \cite{Eisenhart62, RogersSchief02}, the theory of conjugate lattices \cite{BobenkoSuris09,DoliwaSantiniManas00} and the hexahedron recurrence which has been proposed in the context of cluster algebras and dimer configurations \cite{KenyonPemantle13}. The latter connection is made by interpreting the $M$-system as discrete evolution equations for the minors of a matrix, which, in turn, provides a link with the `principal minor assignment problem'~\cite{HoltzSturmfels07}. 

In Section 4, we confine ourselves to the case of a 5$\times$5 matrix which depends on three discrete independent variables. On use of the classical correspondence between lines in a three-dimensional projective space and points in the four-dimensional Pl\"ucker quadric \cite{Plucker65,OnishchikSulanke06},  we demonstrate that the corresponding $M$-system governs privileged integrable line complexes in $\C\P^3$. These {\em fundamental line complexes}, which are termed rectilinear congruences in \cite{DoliwaSantiniManas00}, are characterised by the property of intersecting neighbouring lines and a particular planarity property of the points of intersection. It is noted that the planarity property holds automatically in the case of the afore-mentioned line complexes in $\C\P^4$. It turns out that all fundamental line complexes are encoded in the discrete $M$-system. In order to show this, we make use of a geometric construction of fundamental line complexes which owes its existence to Desargues' classical theorem of projective geometry \cite{Coxeter87}. In fact, this construction reveals that any ``elementary cube'' of a fundamental line complex should be regarded as being embedded in a classical spatial point-line configuration $(15_4\,20_3)$ of 15 points and 20 lines \cite{Baker22}. It is observed that some of the theorems about fundamental line complexes set down in this section turn out to be important in the construction of supercyclidic nets \cite{BobenkoHuhnenVenedeyRoerig14}.

In Section 5, we conclude the paper by characterising fundamental line complexes in terms of correlations \cite{OnishchikSulanke06,SempleKneebone52} of the ambient projective space $\C\P^3$ which interchange ``opposite'' lines of any elementary cube of a fundamental line complex. This characterisation is based on a known theorem (cf.\ \cite{Carver05}) which states that any generic hexagon in $\C\P^3$ uniquely defines an involutive correlation, namely a polarity, which maps any edge to its ``opposite'' counterpart.

\section{A 3D discrete integrable system}

The geometries presented in this paper are integrable in that they are multi-dimensionally consistent in the sense of integrable discrete differential geometry~\cite{BobenkoSuris09}. In algebraic terms, this is readily verified once one has established that their algebraic incarnations constitute canonical reductions of a fundamental system of discrete integrable equations. Thus, we consider three finite sets of ``lower'' and ``upper'' indices
\bela{G1}
L = \{1,\ldots,N\},\quad U^{\rm l}\supset L,\quad U^{\rm r}\supset L
\ela
and associated scalar functions
\bela{G2}
 \bear{c}
  M^{ik} : \Z^N\rightarrow \C,\quad i\in U^{\rm l},\quad k\in U^{\rm r}\as
 (n_1,\ldots,n_N)\mapsto M^{ik}(n_1,\ldots,n_N).
 \ear
\ela
These obey the system of discrete equations
\bela{G3}
  M^{ik}_l = M^{ik} - \frac{M^{il}M^{lk}}{M^{ll}},\quad l\in L\backslash\{i,k\}
\ela
which we term {\em $M$-system}.
Here, and in the following, we suppress the arguments of any discrete function $f$ and indicate increments of the independent variables by lower indices. For instance, if $N=3$ then
\bela{G4}
  f = f(n_1,n_2,n_3),\quad f_1 = f(n_1+1,n_2,n_3),\quad f_{23} = f(n_1,n_2+1,n_3+1).
\ela
It is easy to see that the above system of discrete equations is consistent since the compatibility conditions ${(M_l^{ik})}_m = {(M_m^{ik})}_l$ are satisfied so that the Cauchy data for the $M$-system are \mbox{2-dimensional}. For the same reason, the $M$-system may be extended consistently to a system of equations of the same type on any larger lattice $\Z^{\tilde{N}}$ with
\bela{G5}
  \tilde{U}^{{\rm l},{\rm r}} = U^{{\rm l},{\rm r}}\cup \{N+1,\ldots,\tilde{N}\}
\ela
so that, for fixed $N+1,\ldots,\tilde{N}$, one obtains solutions of the original system (\ref{G3}) and an increment of any of the variables $N+1,\ldots,\tilde{N}$ corresponds to a B\"acklund transformation \cite{BobenkoSuris09,RogersSchief02} of (\ref{G3}).  Thus, the $M$-system constitutes a multi-dimensionally consistent 3D (integrable) system. It is noted that, in order to avoid a re-labelling of the upper indices on the functions $M^{ik}$, we have assumed without loss of generality that $N+1,\ldots,\tilde{N}\not\in U^{{\rm l},{\rm r}}$. 

As indicated below, there exists a variety of connections with the geometric and algebraic theory of integrable systems both discrete and continuous which illustrates the universal nature of the $M$-system (\ref{G3}).

\subsection{The Darboux and Fundamental transformations}

It is well-known that composition of the classical Darboux transformation \cite{Darboux82} and its adjoint leads to a compact binary Darboux transformation formulated in terms of ``squared eigenfunctions'' \cite{MatveevSalle91}. Thus, if $(\phi^k,\phi^l)$ and $(\psi^i,\psi^l)$ are pairs of solutions of the time-dependent Schr\"odinger equation and its adjoint
\bela{G6}
  \phi_t = \phi_{xx} + u\phi,\quad -\psi_t = \psi_{xx} + u\psi,
\ela
where $u=u(x,t)$ constitutes a given potential, then, up to additive constants of integration, bilinear potentials $M^{\alpha\beta}$ are uniquely defined by the compatible pairs
\bela{G7}
  M^{\alpha\beta}_x = \psi^\alpha\phi^\beta,\quad M^{\alpha\beta}_t = \psi^\alpha\phi^\beta_x - \psi^\alpha_x\phi^\beta
\ela
with $\alpha\in\{i,l\}$, $\beta\in\{k,l\}$ and the quantities
\bela{G8}
  \phi^k_l = \phi^k - \phi^l\frac{M^{lk}}{M^{ll}},\quad \psi^i_l = \psi^i - \psi^l\frac{M^{il}}{M^{ll}}
\ela
are again solutions of the Schr\"odinger equations (\ref{G6}) corresponding to the new potential
\bela{G9}
  u_l = u + 2(\ln M^{ll})_{xx}.
\ela
The new squared eigenfunction $M^{ik}_l$ associated with this pair is then readily verified to be \cite{OevelSchief93}
\bela{G10}
  M^{ik}_l = M^{ik} - \frac{M^{il}M^{lk}}{M^{ll}}
\ela
up to a constant of integration. It is important to stress that, algebraically, there is no difference between the transformation formulae (\ref{G8}) and (\ref{G10}). Indeed, if one regards the (adjoint) eigenfunctions $\phi^\beta$ and $\psi^\alpha$ as carrying a ``hidden'' index and sets $M^{0\beta}=\phi^\beta$, $M^{\alpha0}=\psi^\alpha$ then (\ref{G8}) is precisely of the same form as (\ref{G10}).

The algebraic structure of the superposition formula (\ref{G10}) coincides with that of the discrete dynamical system (\ref{G3}). In fact, if one regards the lower index $l$ in (\ref{G10}) as a shift on a lattice and iterates the above binary Darboux transformation then one may generate solutions of system (\ref{G3}). The transformation formulae (\ref{G8}), (\ref{G10}) are universal in that they apply to a large class of linear evolution equations in 1+1 dimensions with potentials $M^{\alpha\beta}$ being defined by suitable analogues of the pair (\ref{G7}) (see, e.g., \cite{OevelSchief93}). Moreover, if one considers vector-valued solutions of the hyperbolic equation
\bela{G11}
  \phi_{xy} = a\phi_x + b\phi_y
\ela
then (\ref{G8})$_1$ represents nothing but the classical Fundamental transformation \cite{Eisenhart62,RogersSchief02} for conjugate nets encoded in (\ref{G11}). The transformation formulae (\ref{G8})$_1$ and (\ref{G10}) have also been shown to apply in the case of the canonical discrete analogue of the Fundamental transformation \cite{DoliwaSantiniManas00,ManasDoliwaSantini97}.

\subsection{Conjugate lattices}

We now consider the case $U^{\rm l} = \{0,1,\ldots,N\}$, $U^{\rm r} = \{1,\ldots,N+d\}$ and introduce the vector notation
\bela{G12}
 \bear{rl}
  \br = & (M^{0k})_{k=N+1,\dots,N+d}\as
  \bM^i = & (M^{ik})_{k=N+1,\ldots,N+d},\quad i=1,\ldots,N.
 \ear
\ela
The $M$-system (\ref{G3}) then translates into 
\bela{G13}
  \br_l = \br - \frac{M^{0l}}{M^{ll}}\bM^l,\quad\bM^i_l = \bM^i - \frac{M^{il}}{M^{ll}}\bM^l.
\ela
Hence, the quantities $\bM^l$ may be regarded as the ``tangent vectors'' of an \mbox{$N$-dimensional} quadrilateral lattice
\bela{G14}
  \br : \Z^N\rightarrow\C^d
\ela
in a $d$-dimensional Euclidean space. Moreover, system (\ref{G13})$_2$ implies that the quadrilaterals are planar and, hence, $\br$ constitutes a conjugate lattice \cite{BobenkoSuris09}. Importantly, it may be shown that all conjugate lattices may be obtained in this manner (cf.\ Section 4.2.3). In view of the previous section, the appearance of conjugate lattices is consistent with the classical and well-known observation that iteration of the classical Fundamental transformation generates planar quadrilaterals \cite{Bianchi2327} and therefore conjugate lattices. Thus, one may regard system (\ref{G3}) as the algebraic essence of the classical Fundamental transformation without reference to geometry. Important algebraic properties of the fundamental $M$-system are presented below.

\section{Minors, \boldmath$\tau$-function and the hexahedron recurrence}

It turns out that the link between the algebraic and geometric perspectives adopted in this paper is provided by algebraic identities for the minors of the matrix $\cal M$ composed of the scalars~$M^{ik}$. Thus, for any two multi-indices \mbox{$A=(a_1\cdots a_s)$} and $B=(b_1\cdots b_s)$, where the entries of each multi-index are assumed to be distinct elements of $U^{\rm l}$ and $U^{\rm r}$ respectively, we may define the minors
\bela{G14a}
  M^{A,B} = \det (M^{a_\alpha b_\beta})_{\alpha,\beta=1,\ldots, s}
\ela
of $\calM$ with $M^{\emptyset,\emptyset}=1$.

\subsection{Jacobi-type identities and identification of a metric}

In terms of the above minors, Jacobi's classical identity for determinants \cite{Hirota03} may be expressed as
\bela{G14c}
M^{A,B}M^{a\bar{a}A,b\bar{b}B} - M^{aA,bB}M^{\bar{a}A,\bar{b}B} + M^{\bar{a}A,bB}M^{aA,\bar{b}B} = 0.
\ela
A compact form of this quadratic identity is given by 
\bela{G14d}
  \langle\Wsf^{a\bar{a},b\bar{b}},\Wsf^{a\bar{a},b\bar{b}}\rangle = 0,
\ela
where 
\bela{G14e}
\Wsf^{a\bar{a},b\bar{b}} = (M^{A,B}, M^{a\bar{a}A,b\bar{b}B},M^{aA,bB},M^{\bar{a}A,\bar{b}B},M^{\bar{a}A,bB},M^{aA,\bar{b}B})
\ela
and the inner product is taken with respect to the block-diagonal metric
\bela{G14f}
  \diag\left[\left(\bear{cc}0&1\\1&0\ear\right),-\left(\bear{cc}0&1\\1&0\ear\right),\left(\bear{cc}0&1\\1&0\ear\right)\right].
\ela
In the above, it is assumed that all upper-left indices and all upper-right indices are distinct so that the associated minors are well defined. 

The identity
\bela{G14g}
 \bear{rl}
  &M^{A,B}M^{\hat{a}a\bar{a}A,\hat{b}b\bar{b}B} + M^{a\bar{a}A,b\bar{b}B}M^{\hat{a}A,\hat{b}B}\as
 -&M^{aA,bB}M^{\hat{a}\bar{a}A,\hat{b}\bar{b}B} - M^{\bar{a}A,\bar{b}B}M^{\hat{a}aA,\hat{b}bB}\as
 +&M^{\bar{a}A,bB}M^{\hat{a}aA,\hat{b}\bar{b}B} + M^{aA,\bar{b}B}M^{\hat{a}\bar{a}A,\hat{b}bB} = 0
\ear
\ela
constitutes an algebraic consequence of the Jacobi identity and proves equally important in the subsequent analysis. Its validity is readily verified by formulating it as
\bela{G14h}
  \langle\Wsf^{a\bar{a},b\bar{b}},\Wsf^{\hat{a}a\bar{a},\hat{b}b\bar{b}}\rangle = 0,
\ela
where
\bela{G14i}
\Wsf^{\hat{a}a\bar{a},\hat{b}b\bar{b}} = (M^{\hat{a}A,\hat{b}B}, M^{\hat{a}a\bar{a}A,\hat{b}b\bar{b}B},M^{\hat{a}aA,\hat{b}bB},M^{\hat{a}\bar{a}A,\hat{b}\bar{b}B},M^{\hat{a}\bar{a}A,\hat{b}bB},M^{\hat{a}aA,\hat{b}\bar{b}B}).
\ela
Indeed, four applications of the Jacobi identity (\ref{G14c}) show that the vector
\bela{G14j}
 \Delta\Wsf = M^{\hat{a}A,\hat{b}B}\Wsf^{a\bar{a},b\bar{b}} - M^{A,B}\Wsf^{\hat{a}a\bar{a},\hat{b}b\bar{b}}
\ela
is of the form
\bela{G14k}
  \Delta\Wsf = (0, *, M^{aA,\hat{b}B}M^{\hat{a}A,bB},M^{\bar{a}A,\hat{b}B}M^{\hat{a}A,\bar{b}B},M^{\bar{a}A,\hat{b}B}M^{\hat{a}A,bB},M^{aA,\hat{b}B}M^{\hat{a}A,\bar{b}B})
\ela
so that it is seen that
\bela{G14l}
 \langle\Delta\Wsf,\Delta\Wsf\rangle = 0.
\ela
The latter is equivalent to (\ref{G14h}) since  $\Wsf^{a\bar{a},b\bar{b}}$ and $\Wsf^{\hat{a}a\bar{a},\hat{b}b\bar{b}}$ are null vectors with respect to the metric (\ref{G14f}). 

A degenerate case of the identity (\ref{G14g}) is obtained by (temporarily) assuming that the rows of the matrix $\calM$ labelled by $a$ and $\hat{a}$ are identical. In this case, three of the terms in (\ref{G14g}) vanish and we are left with
\bela{G14m}
 M^{a\bar{a}A,b\bar{b}B}M^{aA,\hat{b}B} - M^{aA,bB}M^{a\bar{a}A,\hat{b}\bar{b}B} + M^{aA,\bar{b}B}M^{a\bar{a}A,\hat{b}bB} = 0.
\ela
In the same manner as above, it may now be shown that
\bela{G14n}
  \langle\Wsf^{\hat{a}a\bar{a},\hat{b}b\bar{b}},\Wsf^{\hat{a}a\bar{a},\tilde{b}b\bar{b}}\rangle = 0.
\ela
Here, (\ref{G14m}) plays the role of the Jacobi identity (\ref{G14c}). It is evident that, for reasons of symmetry, the identity 
\bela{G14o}
  \langle\Wsf^{\hat{a}a\bar{a},\hat{b}b\bar{b}},\Wsf^{\tilde{a}a\bar{a},\hat{b}b\bar{b}}\rangle = 0
 \ela
likewise holds. The geometric significance of the identities (\ref{G14d}), (\ref{G14h}) and (\ref{G14n}), (\ref{G14o}) is addressed in Section 4.

\subsection{Evolution of the minors and the \boldmath $\tau$-function}

For any $l\in L$ which is not contained in two multi-indices $A$ and $B$, a Laplace expansion of $M^{lA,lB}$ with respect to the row or column labelled by $l$ shows that
\bela{G14b}
  M^{A,B}_l = \frac{M^{lA,lB}}{M^{l,l}}.
\ela
Hence, if $A$ and $B$ constitute simple indices then the $M$-system (\ref{G3}) is retrieved. Accordingly, system (\ref{G3}) may be reinterpreted as the integrable evolution (\ref{G14b}) of the minors of the matrix $\calM$. It is interesting to note that if one sets aside the genesis of the system (\ref{G14b}) then (\ref{G14b}) still constitutes a consistent system for some quantities $M^{A,B}$ since the compatibility condition 
${(M^{A,B}_l)}_m = {(M^{A,B}_m)}_l$ is readily seen to be satisfied. Hence, (\ref{G14a}) may be regarded as {\em a} realisation of the quantities $M^{A,B}$ in terms of the minors of a matrix $\calM$ which evolves according to (\ref{G3}). Whether there exist other realisations is the subject of current research.

In the particular case $U^{\rm l} = U^{\rm r} = L$, the above evolution of the minors gives rise to a compact formulation of the evolution of a $\tau$-function associated with the $M$-system. Thus, it is readily verified that the compatibility conditions ${(\tau_i)}_k = {(\tau_k)}_i$ which guarantee the existence of a function $\tau$ defined according to
\bela{G14p}
  \tau_i = M^{ii}\tau
\ela
are satisfied modulo the $M$-system (\ref{G3}). In fact, it turns out that
\bela{G14q}
  \tau_{ik} = M^{ik,ik}\tau = \left|\bear{cc}M^{ii} & M^{ik}\\ M^{ki} & M^{kk}\ear\right|\tau,
\ela
which is indeed symmetric in the distinct indices $i$ and $k$. In general, if we consider the multi-index $A=(a_1\cdots a_\alpha)$ with distinct entries then the evolution (\ref{G14b}) immediately implies that
\bela{G14r}
  \tau_A = M^{A,A}\tau.
\ela
In particular, three distinct unit increments result in
\bela{G14s}
  \tau_{ikl} = \left|\bear{ccc} M^{ii} & M^{ik} & M^{il}\\ M^{ki} & M^{kk} & M^{kl}\\ M^{li} & M^{lk} & M^{ll}\ear\right|\tau.
\ela
It is noted that any binary Darboux transformation (continuous or discrete) which admits the superposition principle (\ref{G10}) gives rise to determinantal formulae of the above type if the application of a binary Darboux transformation is regarded as a shift on a lattice (see, e.g., \cite{PHD,AratynNissimovPacheva98,Manas01}). 

\subsection{The hexahedron recurrence}

We conclude this section with an application of the $\tau$-function. A so-called `hexahedron recurrence' has been proposed by Kenyon and Pemantle \cite{KenyonPemantle13} which admits a natural interpretation in terms of cluster algebras and dimer configurations \cite{GoncharovKenyon13}. Here, it is demonstrated that the hexahedron recurrence is but another avatar of the $M$-system for $U^{\rm l} = U^{\rm r} = L = \{1,2,3\}$. Thus, four functions 
\bela{Z1}
  h,h^x,h^y,h^z:\Z^3\rightarrow\C
\ela
are said to satisfy the hexahedron recurrence if these are solutions of the coupled system of difference equations
\bela{Z2}
  \bear{rl}
   h_1^xh^xh = & h^xh^yh^z + h_1h_2h_3 + hh_1h_{23}\as
   h_2^yh^yh = & h^xh^yh^z + h_1h_2h_3 + hh_1h_{13}\as
   h_3^zh^zh = & h^xh^yh^z + h_1h_2h_3 + hh_1h_{12}\as
   h_{123}h^2h^xh^yh^z = & (h^xh^yh^z)^2 + h^xh^yh^z(2h_1h_2h_3 + hh_1h_{23} + hh_2h_{13} + hh_3h_{12})\as
&\mbox{}+ (h_1h_2 + hh_{12})(h_1h_3 + hh_{13})(h_2h_3 + hh_{23}).
  \ear
\ela
We now introduce the quantities
\bela{Z3}
  \Delta^x = h_2h_3 + hh_{23},\quad \Delta^y = h_1h_3 + hh_{13},\quad \Delta^z = h_1h_2 + hh_{12}
\ela
which are seen to constitute key components in the system (\ref{Z2}). In particular, the third-order equation (\ref{Z2})$_4$ may be cast into the form
\bela{Z4}
  \frac{h_{123}}{h} = \frac{h^xh^yh^z}{h^3} + \frac{h_1\Delta^x}{h^3} + \frac{h_2\Delta^y}{h^3} + \frac{h_3\Delta^z}{h^3} - \frac{h_1h_2h_3}{h^3} + \frac{\Delta^x\Delta^y\Delta^z}{h^3h^xh^yh^z}
\ela
which suggests that its right-hand side may be written as a 3$\times$3 determinant. Accordingly, we make the change of variables
\bela{Z5}
  M^{23} = -\frac{h^x}{h},\quad M^{31} = -\frac{h^y}{h},\quad M^{12} = -\frac{h^z}{h}
\ela
and express one and two distinct unit increments of the arguments of $h$ in terms of the functions
\bela{Z6}
  M^{11} = \frac{h_1}{h},\quad M^{22} = \frac{h_2}{h},\quad M^{33} = \frac{h_3}{h}
\ela
and
\bela{Z7}
 M^{32} = -\frac{\Delta^x}{hh^x},\quad  M^{13} = -\frac{\Delta^y}{hh^y},\quad M^{21} = -\frac{\Delta^z}{hh^z}
\ela
respectively. The latter two sets of relations may be formulated as
\bela{Z8}
  h_i = M^{ii}h,\quad h_{ik} = -\left|\bear{cc}M^{ii} & M^{ik}\\ M^{ki} & M^{kk}\ear\right|h
\ela
and the hexahedron recurrence relation (\ref{Z4}) becomes
\bela{Z9}
h_{123} = -\left|\bear{ccc} M^{11} & M^{12} & M^{13}\\ M^{21} & M^{22} & M^{23}\\ M^{31} & M^{32} & M^{33}\ear\right|h.
\ela
On use of the definitions (\ref{Z5}), the remaining hexahedron recurrence relations (\ref{Z2})$_{1,2,3}$ may be written as
\bela{Z10}
  M^{23}_1 = M^{32} - \frac{M^{31}M^{12}}{M^{11}}
\ela
and its analogues obtained by cyclically interchanging the indices. Moreover, comparison of (\ref{Z8})$_1$ shifted in the direction $n_k$ and (\ref{Z8})$_2$ yields
\bela{Z11}
  M^{ii}_k = -M^{ii} + \frac{M^{ik}M^{ki}}{M^{kk}},
\ela
while (\ref{Z8})$_2$ shifted in the direction $n_l$, $l\neq i,k$ and matched with (\ref{Z9}) produces
\bela{Z10a}
  M^{32}_1 = M^{23} - \frac{M^{21}M^{13}}{M^{11}}
\ela
and its two counterparts. Hence, the functions $M^{ik}$ evolve according to
\bela{Z11a}
  M^{ik}_l = M^{ki} - \frac{M^{kl}M^{li}}{M^{ll}},\quad i\neq k\neq l\neq i.
\ela
Finally, if the ``transposition'' operator $\mathcal{C}$ is defined by
\bela{Z12}
  \mathcal{C}M^{ik} = M^{ki}
\ela
then the successive change of variables
\bela{Z13}
 \bear{rll}
  M^{ik} &\rightarrow \mathcal{C}^nM^{ik},&\quad n= n_1+n_2+n_3\as
  M^{ik} & \rightarrow (-1)^{n_i+n_k}M^{ik},&\quad (i,k)\in\{(1,2),(2,3),(3,1)\}\as
  M^{ii} & \rightarrow (-1)^{n_k+n_l}M^{ii},&\quad  i\neq k\neq l\neq i
 \ear
\ela
transforms the system (\ref{Z11}), (\ref{Z11a}) into the $M$-system (\ref{G3}) for $U^{\rm l} = U^{\rm r} = L = \{1,2,3\}$ and the function
\bela{Z14}
  \tau = (-1)^{n_1n_2+n_2n_3+n_3n_1} h
\ela
constitutes the corresponding $\tau$-function obeying the linear system (\ref{G14p}) and its higher-order implications (\ref{G14q}) and (\ref{G14s}).

One of the intriguing properties of the hexahedron recurrence (\ref{Z2}) is that the solution remains positive if the initial data are positive. Unfortunately, this important property is hidden in the $M$-system avatar. However, the advantage of the latter formulation is that multi-dimensional consistency is automatically guaranteed. In this connection, it is noted that the cube recurrence, that is, the discrete BKP (Miwa) equation, may be formulated in such a manner that positivity is guaranteed but this property does not hold on higher-dimensional lattices \cite{AdlerBobenkoSuris03}. 

The hexahedron recurrence admits a reduction to a single third-order ``recurrence'' obtained by Kashaev \cite{Kashaev96} in the context of star-triangle moves in the Ising model. The existence of this recurrence is readily verified if one exploits its formulation in terms of the $M$-system. Indeed, it is evident that the $M$-system considered above admits the reduction $M^{ik}=M^{ki}$ corresponding to a symmetric matrix $\calM$. Elimination of its entries $M^{ik}$ between the linear system (\ref{G14q}) and (\ref{G14s}) has been shown to lead to the discrete CKP (dCKP) equation \cite{Schief03b}
\bela{Z15}
(\tau\tau_{123} + \tau_1\tau_{23} - \tau_2\tau_{13} - \tau_3\tau_{12})^2 - 4(\tau_{12}\tau_{13} - \tau_1\tau_{123})(\tau_2\tau_3 - \tau\tau_{23}) = 0.
\ela
By construction, the left-hand side is symmetric in the indices and, remarkably, turns out to be Cayley's 2$\times$2$\times$2 hyperdeterminant \cite{GelfandKapranovZelevinsky94}. Furthermore, the dCKP equation interpreted as a local relation between the principal minors of a symmetric matrix $\calM$ has been derived as a characteristic property by Holtz and Sturmfels \cite{HoltzSturmfels07} in connection with the `principal minor assignment problem'. This remarkable interpretation naturally places this equation on a multi-dimensional lattice and implies its multi-dimensional consistency \cite{TsarevWolf09}.

\section{Integrable discrete line complexes in \boldmath $\C\P^3$}



In the remainder of this paper, we are concerned with the case $N=3$ and \mbox{$U^{\rm l}=U^{\rm r}=\{1,2,3,4,5\}$}. Hence, the associated $M$-system (\ref{G3}) governs the evolution of the matrix
\bela{G15}
  \calM = \left(\bear{ccccc} M^{11}&M^{12}&M^{13}&M^{14}&M^{15}\\
                            M^{21}&M^{22}&M^{23}&M^{24}&M^{25}\\
                            M^{31}&M^{32}&M^{33}&M^{34}&M^{35}\\
                            M^{41}&M^{42}&M^{43}&M^{44}&M^{45}\\
                            M^{51}&M^{52}&M^{53}&M^{54}&M^{55}
  \ear\right)
\ela
as a function of $n_1,n_2,n_3$. More precisely, if we prescribe the Cauchy data
\bela{G16}
  M^{ik}(S^{ik}),\quad S^{ik} = \{(n_1,n_2,n_3) :  n_l=0,\, l\not\in\{i,k\}\}
\ela
at the origin and on appropriate coordinate lines and planes $S^{ik}$ then $\calM$ is uniquely determined by the evolution equations (\ref{G3}).

\subsection{From algebra to geometry}

In order to reveal the geometry encoded in the matrix $\calM$, we focus on the sub-matrix
\bela{G17}
  \hat{\calM} = \left(\bear{cc} M^{44}&M^{45}\\ M^{54}&M^{55}\ear\right)
\ela
which is uniquely determined by its value at the origin and the Cauchy data associated with the remaining matrix entries. 

\subsubsection{Identification of a Pl\"ucker quadric}

If we define the vector-valued function
\bela{G18}
  \Vsf = (M^{\emptyset,\emptyset},M^{45,45},M^{4,4},M^{5,5},M^{5,4},M^{4,5})
\ela
which is composed of all minors of the matrix $\hat{\calM}$ then the functional relationship 
\bela{G19}
  M^{\emptyset,\emptyset}M^{45,45} - M^{4,4}M^{5,5} + M^{5,4}M^{4,5} = 0 
\ela
between these minors may be expressed as
\bela{G20}
  \langle\Vsf,\Vsf\rangle = 0.
\ela
The latter elementary identity is a special case of the Jacobi identity (\ref{G14d}), wherein $A=B=\emptyset$ and $(a,\bar{a})=(b,\bar{b})=(4,5)$. By virtue of the signature of the metric (\ref{G14f}), it is therefore natural to regard $\Vsf$ as a function
\bela{G22}
  \Vsf : \Z^3 \rightarrow \C^{3,3}
\ela
which represents particular homogeneous coordinates of a lattice of points in a four-dimensional quadric $Q^4$ embedded in a five-dimensional complex projective space $\P(\C^{3,3})$. If we identify the quadric $Q^4$ with the complexification of the classical Pl\"ucker quadric \cite{OnishchikSulanke06} then, on use of the Pl\"ucker correspondence (cf.\ Section 4.2.1), we may interpret any point $[\Vsf(\n)]\in Q^4$ as a line $\lsf(\n)\subset\C\P^3$ in a three-dimensional complex projective space. Here, $[\mathsf{X}]\in\P(\C^{3,3})$ refers to a point in projective space with associated homogeneous coordinates $\mathsf{X}\in\C^{3,3}$. Accordingly, $\Vsf$ may be identified with a {\em (discrete) line complex}
\bela{G22a}
 \lsf : \Z^3\rightarrow \{\mbox{lines in $\C\P^3$}\},
\ela
that is, a three-parameter family of lines which are combinatorially attached to the vertices of $\Z^3$ as indicated schematically in figure \ref{linecomplex}.
\begin{figure}
\centerline{\includegraphics[scale=0.5]{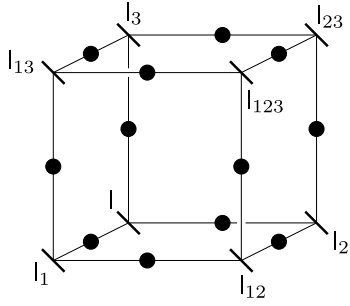}}
\caption{A combinatorial picture of a line complex $\lsf$ which admits the ``intersection property''. The lines are combinatorially attached to the vertices of a $\Z^3$ lattice. Any two neighbouring lines intersect in a point which may be associated with the corresponding edge.}
\label{linecomplex}
\end{figure} 

\subsubsection{Incidence of lines}

In order to uncover the geometric properties of the line complex $\lsf$ defined by $\Vsf$, we first observe that a shift of $\Vsf$ in any lattice direction may be expressed elegantly in terms of higher-order minors of $\calM$. Indeed,``raising indices'' by means of the identity (\ref{G14b}) produces
\bela{G23}
  \Vsf_l = (M^{l,l},M^{l45,l45},M^{l4,l4},M^{l5,l5},M^{l5,l4},M^{l4,l5})/M^{l,l}
\ela
so that setting $A=B=\emptyset$ and $(\hat{a},a,\bar{a}) = (\hat{b},b,\bar{b})=(l,4,5)$ in (\ref{G14h}) yields
\bela{G24}
 \langle \Vsf_l,\Vsf\rangle = 0.
\ela
In geometric terms, this orthogonality condition expresses the fact that the associated lines $\lsf$ and $\lsf_l$ intersect. Thus, the line complex $\lsf$ has the property that the two lines which are combinatorially attached to the two vertices of any edge of $\Z^3$ intersect as depicted in figure \ref{linecomplex}.

The above geometric property implies that the four lines $\lsf,\lsf_l,\lsf_m$ and $\lsf_{lm}$ with $l\neq m$ form a (skew) quadrilateral (see figure \ref{diagonals}).
\begin{figure}
\centerline{\includegraphics[scale=0.5]{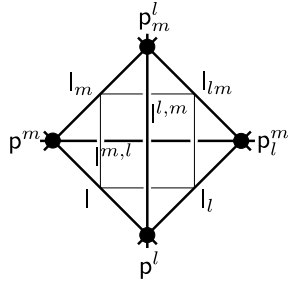}}
\caption{The relationship between the lines $\lsf$, the points of intersection $\psf^l$ and the diagonals $\lsf^{m,p}$.}
\label{diagonals}
\end{figure}
Its diagonals $\lsf^{l,m}$ and $\lsf^{m,l}$ may be obtained by solving the linear system
\bela{G25}
  \langle\Vsf^{*,*},\Vsf\rangle = \langle\Vsf^{*,*},\Vsf_l\rangle = \langle\Vsf^{*,*},\Vsf_m\rangle = \langle\Vsf^{*,*},\Vsf_{lm}\rangle = 0
\ela
subject to the condition $\langle\Vsf^{*,*},\Vsf^{*,*}\rangle=0$. On raising indices by means of (\ref{G14b}), this linear system may be formulated as
\bela{G26}
  \langle\Vsf^{*,*},\Vsf\rangle = \langle\Vsf^{*,*},\Vsf^{l,l}\rangle = \langle\Vsf^{*,*},\Vsf^{m,m}\rangle = \langle\Vsf^{*,*},\Vsf^{lm,lm}\rangle = 0,
\ela
where
\bela{G27}
 \Vsf^{C,D} = (M^{C,D},M^{C45,D45},M^{C4,D4},M^{C5,D5},M^{C5,D4},M^{C4,D5})
\ela
for any multi-indices $C$ and $D$ such that the relevant minors are well defined. Up to scaling, the two null vector solutions turn out to be
\bela{G28}
  \Vsf^{l,m},\quad \Vsf^{m,l}.
\ela
Indeed, for instance, the identity (\ref{G14h}) for $(a,\bar{a})=(b,\bar{b})=(4,5)$ and $\hat{a}=l$, $\hat{b}=m$ with \mbox{$A=B=\emptyset$} and $\hat{a}=m$, $\hat{b}=l$ with $A=l,\,B=m$ coincides with $\langle\Vsf^{l,m},\Vsf\rangle=0$ and \mbox{$\langle\Vsf^{l,m},\Vsf^{lm,lm}\rangle=0$} respectively. On the other hand, $\langle\Vsf^{l,m},\Vsf^{l,l}\rangle=0$ due to the identity (\ref{G14n}) for \mbox{$(a,\bar{a})=(b,\bar{b})=(4,5)$} and $\hat{a}=l,\,\hat{b}=l,\, \tilde{b}=m$ with $A=B=\emptyset$. Similarly, the remaining vanishing inner product $\langle\Vsf^{l,m},\Vsf^{m,m}\rangle=0$ is a consequence of the identity (\ref{G14o}). Finally, for reasons of symmetry, the same arguments apply in the case of the solution $\Vsf^{m,l}$. 

\subsubsection{Coplanarity and concurrency properties}

In order to analyse the properties of the diagonals $\lsf^{l,m}$, we label the points of intersection of any two neighbouring lines by
\bela{G29}
 [\psf^l] = \lsf\cap\lsf_l
\ela
and state the following lemma which will be proven in the next section. It is noted that, for brevity, we make no distinction between a point of intersection $[\psf^l]\in\C\P^3$ and its (particular) homogeneous coordinates $\psf^l\in\C^4$ whenever this cannot lead to any confusion.

\begin{lemma}\label{diagonal}
The diagonal $\lsf^{l,m}$ passes through the points of intersection $\psf^l$ and~$\psf^l_m$.
\end{lemma} 

It is natural to associate the point of intersection of any two neighbouring lines with the edge connecting the corresponding vertices of the $\Z^3$ lattice. Thus, if we consider four edges of the same ``type'' of an elementary cube then, according to the above lemma, the associated points of intersection $\psf^l,\psf^l_m,\psf^l_p$ and $\psf^l_{mp}$ may be regarded as the vertices of a quadrilateral with edges $\lsf^{l,m},\lsf^{l,p}$ and $\lsf^{l,m}_p,\lsf^{l,p}_m$ (cf.~figure \ref{fundamental}). 
\begin{figure}
\centerline{\includegraphics[scale=0.5]{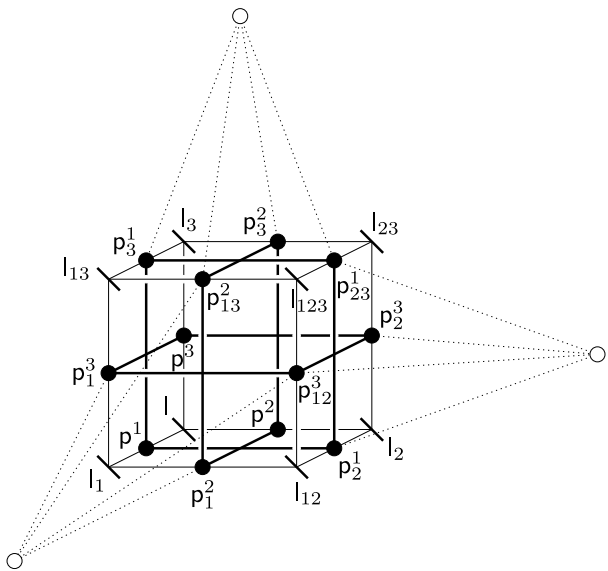}}
\caption{A combinatorial picture of a fundamental line complex $\lsf$. The four points of intersection attached to any four edges of the same ``type'' are coplanar. This is equivalent to the concurrency of any four ``diagonals'' of the same ``type'' indicated by the dotted lines.}
\label{fundamental}
\end{figure} 
Remarkably, this quadrilateral is planar since, for instance, the lines $\lsf^{l,m}$ and $\lsf^{l,m}_p$ intersect as may be concluded from
\bela{G30}
  \langle\Vsf^{l,m},\Vsf^{l,m}_p\rangle \sim \langle\Vsf^{l,m},\Vsf^{pl,pm}\rangle = 0
\ela
by virtue of the identity (\ref{G14h}). Furthermore, since the lines $\lsf^{l,m}$ and $\lsf^{p,m}$ lie in the plane spanned by the lines $\lsf$ and $\lsf_m$, the former lines likewise intersect. Algebraically, this is confirmed by
\bela{G31}
  \langle\Vsf^{l,m},\Vsf^{p,m}\rangle = 0
\ela
which, once again, constitutes a special case of the identity (\ref{G14o}). In fact, it is evident that, under appropriate genericity assumptions, the four diagonals of the same ``type'' $\lsf^{l,m},\lsf^{l,m}_p$ and $\lsf^{p,m},\lsf^{p,m}_l$ must be concurrent as indicated in figure \ref{fundamental}. Hence, we are led to privileged line complexes which have been termed rectilinear congruences in \cite{DoliwaSantiniManas00}. In the current context, we adopt the following terminology.

\begin{definition}
A line complex $\lsf:\Z^3\rightarrow\{\mbox{lines in $\C\P^3$}\}$ is termed {\em fundamental} if any neighbouring lines $\lsf$ and $\lsf_l$ intersect and the points of intersection $\psf^l$ enjoy the {\em coplanarity property}, that is, for any distinct $l,m$ and $p$, the quadrilaterals $(\psf^l,\psf^l_m,\psf^l_{mp},\psf^l_p)$ are planar. Equivalently, the lines passing through the points $\psf^l$ and $\psf^l_m$ admit the {\em concurrency property}, that is, the four lines connecting the points $\psf^l,\psf^p,\psf^l_p,\psf^p_l$ and their shifted counterparts $\psf^l_m,\psf^p_m,\psf^l_{pm},\psf^p_{lm}$ are concurrent.
\end{definition}

\begin{remark}\label{oneimpliesfive}
It is readily verified that if any one of the six coplanarity and concurrency conditions associated with an elementary cube of a line complex is satisfied then the other five conditions automatically hold.
\end{remark}

The analysis presented in the preceding may therefore be summarised as follows.

\begin{theorem}\label{fundamentaltheorem}
Any solution $\mathcal{M}$ of the $M$-system (\ref{G3}) with $N=3$ and $U^{\rm l}=U^{\rm r}=\{1,2,3,4,5\}$ parametrising a function $\Vsf$ according to (\ref{G18}) encapsulates a fundamental line complex $\lsf$ via the Pl\"ucker correspondence $\Vsf\leftrightarrow\lsf$.
\end{theorem}

\subsection{From geometry to algebra}

In this section, we demonstrate that the converse of Theorem \ref{fundamentaltheorem} is also true, that is, all fundamental line complexes are algebraically represented by the $M$-system (\ref{G3}).

\subsubsection{The Pl\"ucker correspondence}

The classical Pl\"ucker correspondence is established by considering a lift $\asf\wedge\bsf$ of a line $\lsf$ in $\C\P^3$, that is, by introducing homogeneous coordinates
\bela{G33}
  \asf = \left(\bear{c}\alpha^0\\ \alpha^1\\ \alpha^2\\ \alpha^3\ear\right),\quad 
  \bsf = \left(\bear{c}\beta^0\\ \beta^1\\ \beta^2\\ \beta^3\ear\right)
\ela
of two points on the line $\lsf$ and setting
\bela{G34}
  \Vsf = (\gamma^{01},\gamma^{23},\gamma^{02},\gamma^{13},\gamma^{03},\gamma^{12}),
\ela
where the coefficients $\gamma^{\mu\nu}$ are defined by the sub-determinants
\bela{G35}
  \gamma^{\mu\nu} = \det\left(\bear{cc} \alpha^\mu & \beta^\mu\\ \alpha^\nu & \beta^\nu\ear\right).
\ela
It is easy to verify that the $\gamma^{\mu\nu}$ obey the classical Pl\"ucker identity
\bela{G36}
  \gamma^{01}\gamma^{23} - \gamma^{02}\gamma^{13} + \gamma^{03}\gamma^{12} = 0
\ela
so that the components of $\Vsf$ indeed constitute homogeneous coordinates of a point in the Pl\"ucker quadric~$Q^4$. Specifically, in the generic case, we may make the choice
\bela{G37}
  \asf = \left(\bear{c}0\\ 1\\ M^{44} \\ M^{54}\ear\right),\quad 
  \bsf = \left(\bear{c}-1\\ 0\\ M^{45}\\ M^{55}\ear\right)
\ela
so that
\bela{G38}
  \Vsf = \left(1,\left|\bear{cc}M^{44}&M^{45}\\ M^{54}&M^{55}\ear\right|,M^{44},M^{55},M^{54},M^{45}\right)
\ela
is precisely of the form (\ref{G18}). However, at this stage, $M^{44},M^{55},M^{54}$ and $M^{45}$ are merely labels of the ``non-trivial'' homogeneous coordinates of the points $\asf$ and~$\bsf$.

We now focus on the lines of a fundamental line complex $\lsf$. Since any neighbouring lines $\lsf$ and $\lsf_l$ intersect, the points $\asf,\bsf$ and $\asf_l,\bsf_l$ must be coplanar. This may be expressed as
\bela{G39}
  N^{l5}\Delta_l\asf = N^{l4}\Delta_l\bsf
\ela
due to the constancy of the first two components of $\asf$ and $\bsf$. Here, $N^{l4}$ and $N^{l5}$ are functions to be determined and $\Delta_lf = f_l - f$. Accordingly, the points of intersection $\psf^l$ are given by
\bela{G40}
  \psf^l = N^{l5}\asf - N^{l4}\bsf.
\ela
As an illustration, we (temporarily) consider a fundamental line complex parame\-trised by a solution of the $M$-system as stated in Theorem \ref{fundamentaltheorem}. In this case, the evolution of the matrix (\ref{G17}) may be formulated as
\bela{G41}
  \Delta_l\asf = -\frac{M^{l4}}{M^{ll}}\Msf^l,\quad \Delta_l\bsf = -\frac{M^{l5}}{M^{ll}}\Msf^l,\quad
  \Msf^l = \left(\bear{c}0\\ 0\\ M^{4l}\\ M^{5l}\ear\right)
\ela
so that
\bela{G42}
  \psf^l\sim M^{l5}\asf - M^{l4}\bsf = \left(\bear{c}M^{l4}\\ M^{l5}\\ M^{l5}M^{44}-M^{l4}M^{45}\\ M^{l5}M^{54}-M^{l4}M^{55}\ear\right) = \left(\bear{c}M^{l,4}\\ M^{l,5}\\ M^{l4,54}\\ M^{l5,54}\ear\right).
\ela
On use of the Pl\"ucker correspondence (\ref{G33})-(\ref{G35}) and the identities of Section 3.1, it is then straightforward to verify that the line passing through the points of intersection $\psf^l$ and $\psf^l_m$ is indeed given by the diagonal $\lsf^{l,m}$ as stated in Lemma~\ref{diagonal}.

\subsubsection{Fundamental line complexes in \boldmath $\C\P^4$}

In order to make the transition from a fundamental line complex to a solution of the $M$-system, we first observe that, as pointed out in \cite{DoliwaSantiniManas00}, the inclusion of the coplanarity property or, equivalently, the concurrency property in the definition of a fundamental line complex is due to the dimensionality of the ambient space of the line complexes discussed here. Thus:

\begin{definition}
A line complex $\lsf:\Z^3\rightarrow\{\mbox{lines in $\C\P^4$}\}$ is termed {\em fundamental} if any neighbouring lines $\lsf$ and $\lsf_l$ intersect.
\end{definition}

Indeed, the two sets of four lines $\{\lsf,\lsf_1,\lsf_2,\lsf_{12}\}$ and $\{\lsf_3,\lsf_{13},\lsf_{23},\lsf_{123}\}$ span two (three-dimensional) hyperplanes which, generically, intersect in a (two-dimensional) plane. The latter plane contains the points $\psf^3,\psf^3_{1},\psf^3_{2},\psf^3_{12}$ so that the set of points $\psf^3$ automatically enjoys the coplanarity property.

It turns out that fundamental line complexes in $\C\P^3$ may be regarded as projections of fundamental line complexes in $\C\P^4$.

\begin{theorem}\label{projectiontheorem}
A line complex in $\C\P^3$ is fundamental if and only if it may be regarded as a projection onto a hyperplane of a fundamental line complex in $\C\P^4$.
\end{theorem}

\begin{proof}
Given a fundamental line complex in $\C\P^3$ which we regard as being embedded in a hyperplane of $\C\P^4$, we consider an ``elementary cube'' of 8 lines $\lsf,\ldots,\lsf_{123}$ and associated points of intersection. We begin by choosing a generic projection $\pi$ and prescribing five (black) points $\hat{\psf}^2_3,\hat{\psf}^3_2,\hat{\psf}^1_2,\hat{\psf}^2_1,\hat{\psf}^3_1$ such that $\pi(\hat{\psf}^l_m) = \psf^l_m$ for $(l,m)\neq(1,3)$ as depicted in figure \ref{projectionproof}. 
\begin{figure}
\centerline{\includegraphics[scale=0.5]{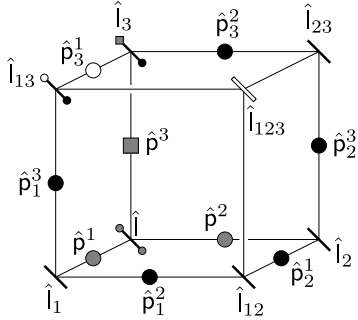}}
\caption{Illustration of the proof that fundamental line complexes in $\C\P^3$ may be regarded as projections of fundamental line complexes in $\C\P^4$.}
\label{projectionproof}
\end{figure}
These points define the four lines $\hat{\lsf}_{23},\hat{\lsf}_2,\hat{\lsf}_{12},\hat{\lsf}_1$ which project onto the lines $\lsf_{23},\lsf_2,\lsf_{12},\lsf_1$ respectively. Since $[\psf^1]\in\lsf_1$ and $[\psf^2]\in\lsf_2$, the lines $\hat{\lsf}_1$ and $\hat{\lsf}_2$ uniquely determine the (grey) points $\hat{\psf}^1$ and $\hat{\psf}^2$ lying on $\hat{\lsf}_1$ and $\hat{\lsf}_2$ respectively and obeying $\pi(\hat{\psf}^1)=\psf^1$ and $\pi(\hat{\psf}^2)=\psf^2$. The line $\hat{\lsf}$ passing through $\hat{\psf}^1$ and $\hat{\psf}^2$ and projecting onto the line $\lsf$ gives rise, in turn, to the point $\hat{\psf}^3$ (grey square) lying on $\hat{\lsf}$ such that $\pi(\hat{\psf}^3)=\psf^3$. The latter point and $\hat{\psf}^2_3$ define the line $\hat{\lsf}_3$ projecting onto $\lsf_3$. Once again, since $[\psf^1_3]\in\lsf_3$, there exists a unique (white) point $\hat{\psf}^1_3$ lying on $\hat{\lsf}_3$ which projects onto $\psf^1_3$. Finally, by construction, the line $\hat{\lsf}_{13}$ passing through $\hat{\psf}^1_3$ and $\hat{\psf}^3_1$ projects onto the line $\lsf_{13}$.

As discussed in Section 4.3, the six lines $\hat{\lsf}_m,\hat{\lsf}_{mp}\subset\C\P^4$ assumed to be in general position uniquely determine a transversal (boxed) line $\hat{\lsf}_{123}$. Since projection preserves the coplanarity property, the 8 lines $\hat{\lsf},\ldots,\hat{\lsf}_{123}$ projected onto the hyperplane form an elementary cube of a fundamental line complex in $\C\P^3$. However, since this elementary cube and the original elementary cube share 7 lines, the two eighth lines must also coincide, that is, $\pi(\hat{\lsf}_{123})=\lsf_{123}$. The assertion that 7 lines of an elementary cube of a fundamental line complex in $\C\P^3$ uniquely determine the eighth line is the content of Theorem \ref{desarguestheorem} which is proven in Section 4.3. It remains to observe that the above arguments also apply (iteratively) to complete fundamental line complexes in light of the Cauchy problems for fundamental line complexes formulated in the same section.
\end{proof}

\subsubsection{A conjugate lattice connection}

For any pair of fundamental line complexes related by a projection in the sense of the above theorem, we identify the associated hyperplane with the ``hyperplane at infinity''. Hence, the transition from a fundamental line complex in $\C\P^3$ to a corresponding fundamental line complex in $\C\P^4$ merely amounts to adding an appropriate fifth homogeneous coordinate so that the coplanarity condition (\ref{G39}) admits the counterpart
\bela{G44}
   N^{l5}\Delta_l\tilde{\asf} = N^{l4}\Delta_l\tilde{\bsf},\quad \tilde{\asf} = \left(\bear{c}0\\ 1\\ M^{44} \\ M^{54}\\ M^{64}\ear\right),\quad 
  \tilde{\bsf} = \left(\bear{c}-1\\ 0\\ M^{45}\\ M^{55}\\ M^{65}\ear\right)
\ela
with the {\em same} functions $N^{l4}$ and $N^{l5}$. The latter may be resolved by introducing discrete ``tangent vectors'' $\bM^l$ according to
\bela{G45}
  \Delta_l\ba = N^{l4}\bM^l,\quad \Delta_l\bb = N^{l5}\bM^l,\quad \ba = \left(\bear{c}M^{44}\\ M^{54}\\ M^{64}\ear\right),\quad \bb = \left(\bear{c}M^{45}\\ M^{55}\\ M^{65}\ear\right).
\ela
Here, we have confined ourselves to the three-dimensional non-trivial part of the coplanarity condition (\ref{G44}). Thus, formally, the vector-valued functions
\bela{G46}
  \ba : \Z^3\rightarrow\C^3,\quad \bb : \Z^3\rightarrow\C^3
\ela
represent two quadrilateral lattices in a complex three-dimensional Euclidean space with the property that corresponding edges are parallel. It is well known that, in the generic case, the quadrilaterals of lattices of this type must be planar and, hence, by definition, the conjugate lattices $\ba$ and $\bb$ constitute Combescure transforms of each other \cite{KonopelchenkoSchief98}.

The planarity of the quadrilaterals of $\ba$ and $\bb$ may be expressed as
\bela{G47}
  \bM^l_m= I^{ml}\bM^l + N^{ml}\bM^m,\quad l\neq m
\ela
with associated compatibility conditions ${(\bM^l_m)}_p={(\bM^l_p)}_m$, leading to
\bela{G48}
 \bear{rl}
    & I^{ml}_p(I^{pl}\bM^l + N^{pl}\bM^p) + N^{ml}_p(I^{pm}\bM^m + N^{pm}\bM^p) \as
 = &  I^{pl}_m(I^{ml}\bM^l + N^{ml}\bM^m) + N^{pl}_m(I^{mp}\bM^p + N^{mp}\bM^m).
 \ear
\ela
Under the non-degeneracy assumption of linearly independent tangent vectors $\bM^1,\bM^2,\bM^3$, we therefore conclude that, in particular,
\bela{G49}
    I^{ml}_p I^{pl} =  I^{pl}_m I^{ml}.
\ela
It is observed that these relations may be interpreted as the algebraic incarnation of Theorem~\ref{projectiontheorem}. Thus, if the vectors $\bM^l$ were associated with a line complex in $\C\P^3$ and therefore two-dimensional then the expansions (\ref{G47}) would still be valid but equating to zero the coefficients multiplying the tangent vectors in (\ref{G48}) would not be justified. The conditions (\ref{G49}) give rise to the parametrisation
\bela{G50}
  I^{ml} = \frac{\varphi^l_m}{\varphi^l}
\ela
in terms of potentials $\varphi^l$. These may be used to scale the coefficients $I^{ml}$ to unity by applying the gauge transformation $\bM^l\rightarrow\varphi^l\bM^l$. Hence, we may assume without loss of generality that $I^{ml}=1$. The remaining compatibility conditions then reduce to the nonlinear system
\bela{G51}
  N^{ml}_p = \frac{N^{ml} + N^{mp}N^{pl}}{1- N^{mp}N^{pm}}.
\ela
The latter constitutes a standard discretisation of the classical Darboux equations governing equivalence classes of Combescure transforms of conjugate coordinate systems \cite{KonopelchenkoBogdanov95,DoliwaSantini97}. For any fixed solution $N^{lm}, \bM^l$ of the discrete Darboux system and the linear system (\ref{G47}), two Combescure transforms are given by $\ba$ and $\bb$, where the coefficients $N^{l4}$ and $N^{l5}$ are solutions of the same linear system
\bela{G52}
  N^{lk}_m=\frac{N^{lk} + N^{lm}N^{mk}}{1 - N^{lm}N^{ml}},\quad k=4,5,
\ela
namely the compatibility conditions associated with the linear system (\ref{G45}). It is noted that the two systems (\ref{G51}) and (\ref{G52}) are identical in form.

The final step in the identification of the $M$-system in question is based on the observation that the discrete Darboux system admits the ``conservation laws''
\bela{G53}
  \Xi^{lm}\Xi^{lp}_m = \Xi^{lp}\Xi^{lm}_p,\quad \Xi^{lm} = 1 - N^{lm}N^{ml}.
\ela
Accordingly, there exist associated potentials $M^{ll}$ defined by the linear equations
\bela{G54}
  M^{ll}_m = (1 - N^{lm}N^{ml})M^{ll}.
\ela
Hence, if we introduce the parametrisation
\bela{G55}
  N^{lm} = -\frac{M^{lm}}{M^{ll}},\quad \bM^l = \left(\bear{c}M^{4l}\\ M^{5l}\\ M^{6l}\ear\right)
\ela
then the systems (\ref{G45}), (\ref{G47}), (\ref{G51}), (\ref{G52}) and (\ref{G54}) may be combined to obtain the $M$-system
\bela{G56}
 \bear{c}\displaystyle
  M^{ik}_l = M^{ik} -\frac{M^{il}M^{lk}}{M^{ll}},\quad l\not\in\{i,k\}\AS
 i\in\{1,2,3,4,5,6\},\quad k\in\{1,2,3,4,5\},\quad l\in\{1,2,3\}.
 \ear
\ela
Since the coefficients $M^{6k}$ are merely auxiliary functions which represent the transition of a fundamental line complex from $\C\P^3$ to $\C\P^4$, we are now in a position to state the converse of Theorem \ref{fundamentaltheorem}.

\begin{theorem}\label{conversetheorem}
Any fundamental line complex in $\C\P^3$ gives rise to a solution $\mathcal{M}$ of the $M$-system (\ref{G3}) with $N=3$ and $U^{\rm l}=U^{\rm r}=\{1,2,3,4,5\}$ via the lift $\asf\wedge\bsf$ of the lines $\lsf$ encapsulated in (\ref{G37}).
\end{theorem}

\subsection{Geometric construction of fundamental line complexes}

An elementary cube of a fundamental line complex $\lsf$ in $\C\P^4$ may be constructed by prescribing a skew hexagon formed by the two triples of lines $\lsf_1,\lsf_2,\lsf_3$ and $\lsf_{12},\lsf_{23},\lsf_{13}$ (cf.\ figure \ref{elementary}). 
\begin{figure}
\centerline{
\begin{minipage}{2cm}
\includegraphics[scale=0.5]{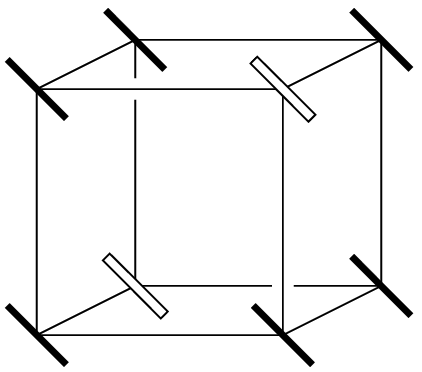}
\end{minipage}\qquad\qquad\qquad
\begin{minipage}{6cm}
\includegraphics[scale=0.5]{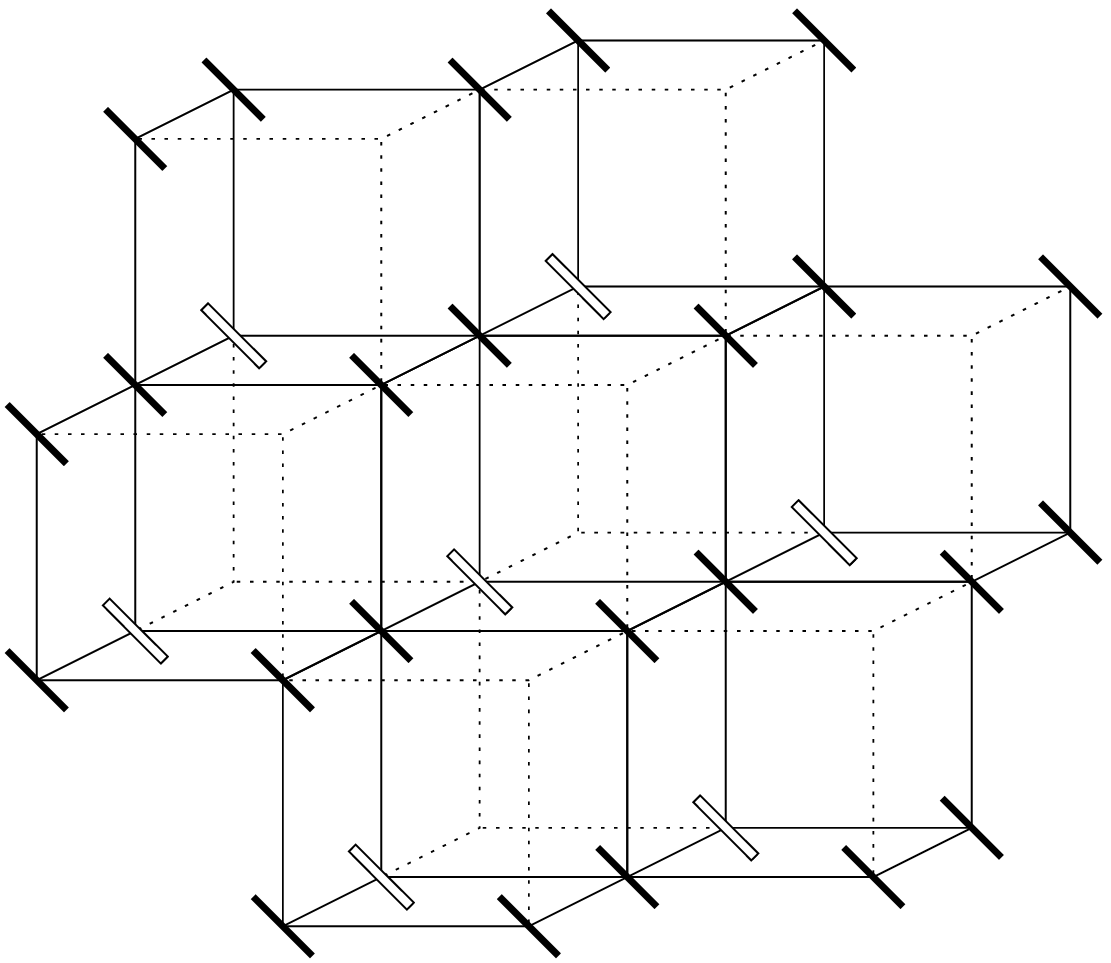}
\end{minipage}
}
\caption{{\sc Left}: Six (black) lines forming a spatial hexagon uniquely determine the remaining two (boxed) lines of an ``elementary cube'' of a fundamental line complex in $\C\P^4$. {\sc Right}: A Cauchy problem for fundamental line complexes. In $\C\P^4$, the boxed lines are determined by the black lines. In $\C\P^3$, the boxed lines constitute additional Cauchy data which have to intersect the relevant triples of black lines.
}
\label{elementary}
\end{figure}
On the assumption that these triples are in general position, the lines $\lsf$ and $\lsf_{123}$ are then uniquely determined by the requirement that these pass through the relevant triple of lines. An entire fundamental line complex is uniquely determined by prescribing a ``plane'' of hexagons as Cauchy data, that is, by specifying the set of lines
\bela{G57}
  \{\lsf(n_1,n_2,n_3) : n_1 + n_2 + n_3 \in\{1,2\}\}
\ela
as illustrated in figure \ref{elementary}.
%
%

In the case of fundamental line complexes in $\C\P^3$, the Cauchy problem becomes non-trivial since the coplanarity property needs to be taken into account. If we prescribe a hexagon of lines as above then the three lines $\lsf_1,\lsf_2,\lsf_3$ may be interpreted as three generators of a unique quadric. The line $\lsf$ then constitutes an element of the second one-parameter family of generators of the quadric. An analogous interpretation is valid in the case of the second triple of lines  $\lsf_{12},\lsf_{23},\lsf_{13}$ and its transversal line $\lsf_{123}$. As shown in \cite{DoliwaSantiniManas00}, the line $\lsf_{123}$ is uniquely determined by the coplanarity property and a fixed choice of the line $\lsf$. In fact, the existence of the line $\lsf_{123}$ may be traced back to the classical Desargues theorem of projective geometry \cite{Coxeter87}.

\begin{theorem}\label{desarguestheorem}
Given seven lines in $\C\P^3$ which are combinatorially attached to seven vertices of an elementary cube and intersect each other ``along edges'', there exists a unique eighth line such that the eight lines constitute an elementary cube of a fundamental line complex.
\end{theorem} 
 
\begin{proof}
Without loss of generality, we consider the seven (short black) lines $\lsf,\lsf_1,\lsf_2,\lsf_3$ and $\lsf_{12},\lsf_{23},\lsf_{13}$ depicted in figure \ref{baker}.
\begin{figure}
\centerline{\includegraphics[scale=0.5]{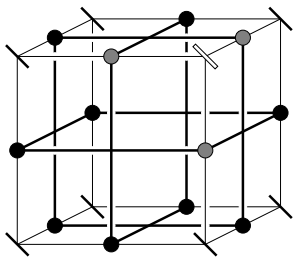}}
\caption{Seven (black) lines in $\C\P^3$ together with the associated 9 (black) points of intersection uniquely determine the remaining three (grey) points by virtue of the planarity property. Desargues' theorem then guarantees that these are collinear so that the remaining eighth (boxed) line completes the elementary cube of a fundamental line complex.}
\label{baker}
\end{figure}
The associated nine (black) points of intersection are given by $\psf^1,\psf^1_2,\psf^1_3$, $\psf^2,\psf^2_1,\psf^2_3$ and $\psf^3,\psf^3_1,\psf^3_2$. The coplanarity property is implemented by defining $\psf^3_{12}$ as the (grey) point of intersection of the plane passing through $\psf^3,\psf^3_1,\psf^3_2$ and the line $\lsf_{12}$. Similarly, the (grey) points $\psf^1_{23}$ and $\psf^2_{13}$ are constructed. In order to demonstrate that the three points $\psf^1_{23}, \psf^2_{13}$ and $\psf^3_{12}$ are indeed collinear and therefore define the (boxed) line $\lsf_{123}$, we focus on a subset of nine lines, namely, for instance, the lines $\lsf_1,\lsf_{12},\lsf_{13}$ and the lines passing through the pairs of points of intersection $(\psf^1,\psf^1_2),(\psf^1_3,\psf^1_{23}),(\psf^3_1,\psf^3_{12})$ and $(\psf^1,\psf^1_3),(\psf^1_2,\psf^1_{23}),(\psf^2_1,\psf^2_{13})$ (cf.\ figure~\ref{desarguesbaker}). Since the coplanarity property is equivalent to the concurrency property, the three lines of each of the second and third triplet of lines are concurrent. 
\begin{figure}
\centerline{\includegraphics[scale=0.5]{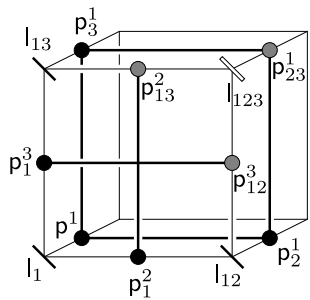}}
\caption{The nine black lines represent a subset of the lines in figure \ref{baker}. The tenth (boxed) line exists due to Desargues' theorem as illustrated in figure \ref{desargues}.}
\label{desarguesbaker}
\end{figure}
Hence, as illustrated in figure \ref{desargues}, the nine lines are part of a spatial $(10_3)$ configuration of 10 points and 10 lines since the tenth line passing through the points $\psf^1_{23}, \psf^2_{13},\psf^3_{12}$ exists by virtue of Desargues' classical theorem.
\end{proof}
\begin{figure}
\centerline{\includegraphics[scale=0.5]{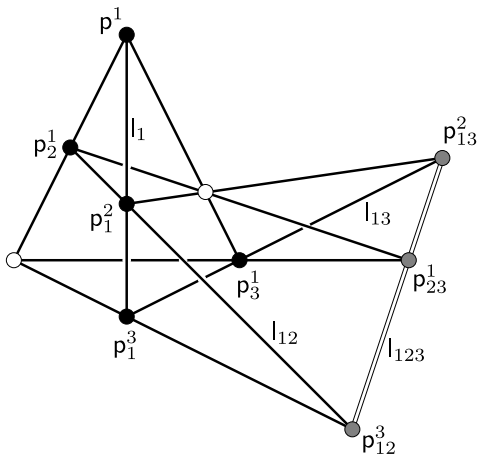}}
\caption{The Desargues configuration associated with the ten lines of the fundamental line complex displayed in figure~\ref{desarguesbaker}. The two white points exist due to the concurrency property.}
\label{desargues}
\end{figure}

\begin{remark}
The above theorem implies that, for any given ``hexagon'' of six lines $\lsf_1,\lsf_2,\lsf_3$ and $\lsf_{23},\lsf_{13},\lsf_{12}$ of an elementary cube of a fundamental line complex, the planarity property gives rise to a unique map between the lines $\lsf$ and $\lsf_{123}$ contained in the hyperboloids defined by $\lsf_1,\lsf_2,\lsf_3$ and $\lsf_{23},\lsf_{13},\lsf_{12}$ respectively.
\end{remark}

\begin{remark}\label{frontispiece}
The complete set of 8 lines of an elementary cube of a fundamental line complex and the 12 associated diagonals together with the 12 points of intersection of the lines and the three points of concurrency of the diagonals give rise to a spatial point-line configuration $(15_4\, 20_3)$ of 15 points and 20 lines with four lines through each point and three points on each line (cf.\ figure~\ref{configuration}). 
\begin{figure}
\centerline{\includegraphics[scale=0.2]{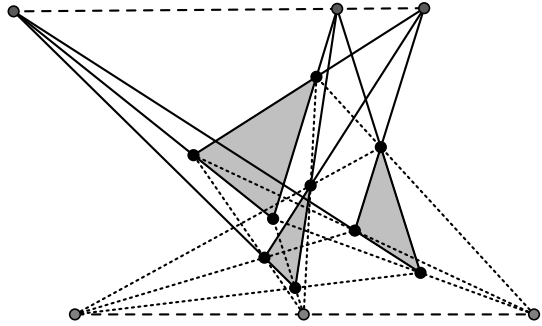}}
\caption{A classical $(15_4\,20_3)$ point-line configuration which is characterised by the property that three triangles are perspective from three collinear points or, equivalently, from a line.}
\label{configuration}
\end{figure} 
This configuration constitutes the frontispiece to Baker's first volume of {\it Principles of Geometry} \cite{Baker22} and was used (but not displayed) in Coxeter's monograph {\it Projective Geometry} \cite{Coxeter87} in connection with the proof of the converse of Desargues' theorem.
\end{remark}

As in the case of fundamental line complexes in $\C\P^4$, iterative application of the construction of elementary cubes leads to a unique fundamental line complex in $\C\P^3$. However, according to the above theorem, the Cauchy data (\ref{G57}) must be extended to the set of lines
\bela{G58}
   \{\lsf(n_1,n_2,n_3) : n_1 + n_2 + n_3 \in\{0,1,2\}\}
\ela
as indicated in figure \ref{elementary}.
Once again, it is understood that the Cauchy data respect the requirement that neighbouring lines intersect.

\section{Fundamental line complexes and correlations}

It turns out that fundamental line complexes in $\C\P^3$ may be characterised in terms of {\em correlations}. A correlation of a $d$-dimensional projective space is an incidence-preserving transformation which maps $k$-dimensional projective subspaces to $d-k-1$-dimensional projective subspaces \cite{OnishchikSulanke06,SempleKneebone52}. In particular, in three dimensions, the points of a line are mapped to planes which meet in a line. Here, we represent a correlation by a map
\bela{G59}
  \kappa :  \C\P^3 \rightarrow \{\mbox{planes in $\C\P^3$}\}.
\ela
For brevity, we use the same symbol $\kappa$ for the representation of this map in terms of homogeneous coordinates. Any correlation is then encoded in a complex 4$\times$4 matrix $\Bsf$ such that
\bela{G60}
  \kappa(\xsf) = \{\ysf\in\C^4 : \ysf^T\Bsf\xsf = 0\}.
\ela

In the previous section, it has been demonstrated that, for any given line of an elementary cube of a fundamental line complex in $\C\P^3$, the ``opposite'' line is determined by the hexagon formed by the remaining six lines and the planarity property. It is therefore natural to inquire as to whether opposite lines of an elementary cube of a fundamental line complex are linked by a common correlation. To this end, we recall the self-polarity of a hexagon in $\C\P^3$  (cf.\ \cite{Carver05}). Thus, we consider a hexagon in general position with vertices $\xsf^1,\ldots,\xsf^6$ as displayed in figure~\ref{hexagon}.
\begin{figure}
\centerline{\includegraphics[scale=0.5]{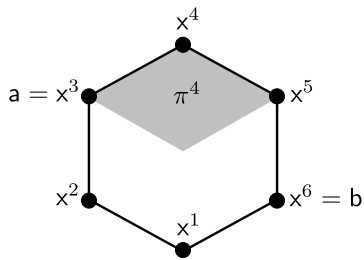}}
\caption{A hexagon in $\C\P^3$ with vertices $\xsf^i$ and associated planes $\pi^i$.}
\label{hexagon}
\end{figure}
The planes spanned by any three successive vertices $\xsf^{i-1},\xsf^i,\xsf^{i+1}$, where indices are taken modulo~6, are denoted by $\pi^i$. The condition for a correlation $\kappa$ to map any line (extended edge) of the hexagon to its opposite line may therefore be expressed as
\bela{G61}
  \kappa(\xsf^i) = \pi^{i+3},\quad i=1,\ldots, 6.
\ela

As in the case of projective transformations of $\C\P^3$ (collineations), a correlation is uniquely determined by the images of five points in general position. Accordingly, a correlation $\kappa$ of the above type is unique if it exists and is necessarily involutive since it acts as an involution on the hexagon. The corresponding matrix $\Bsf$ must then be either skew-symmetric or symmetric. In the former case, any point lies in its image, which contradicts the assumption of the hexagon being in general position. Hence, the correlation constitutes a polarity \cite{OnishchikSulanke06,SempleKneebone52} with respect to the quadric $\xsf^T\Bsf\xsf = 0$ defined by the symmetric matrix $\Bsf$ and it is convenient to define the inner product
\bela{G62}
  \langle\ysf,\xsf\rangle = \ysf^T\Bsf\xsf.
\ela
Thus, the conditions (\ref{G61}) may be formulated as
\bela{G63}
  \langle \xsf^{i+2},\xsf^i\rangle = \langle\xsf^{i+3},\xsf^i\rangle = \langle\xsf^{i+4},\xsf^i\rangle = 0,\quad i=1,\ldots,6.
\ela
The latter constitute 9 linear equations for the 10 coefficients of the symmetric matrix $\Bsf$. 

In order to demonstrate that the above system of linear equations admits a solution (which is unique up to scaling), we assume without loss of generality that
\bela{G64}
  \xsf^1 = \left(\bear{c}0\\0\\1\\0\ear\right),\quad \xsf^2 = \left(\bear{c}1\\0\\0\\0\ear\right),\quad \xsf^4 = \left(\bear{c}0\\0\\0\\1\ear\right),\quad \xsf^5 = \left(\bear{c}0\\1\\0\\0\ear\right).
\ela
The remaining two vertices are in general position and hence admit the parametrisation
\bela{G65}
  \xsf^3 = \asf = \left(\bear{c}\alpha^0\\ \alpha^1\\ \alpha^2\\ \alpha^3\ear\right),\quad   \xsf^6 = \bsf = \left(\bear{c}\beta^0\\ \beta^1\\ \beta^2\\ \beta^3\ear\right)
\ela
with non-vanishing components. The four equations of the linear system (\ref{G63}) which do not involve the vertices $\asf$ and $\bsf$ are readily seen to imply that the matrix $\Bsf$ is of the form
\bela{G66}
  \Bsf = \left(\bear{cccc} B^{00}&0&B^{02}&0\\
                                    0&B^{11}&0&B^{13}\\
                                    B^{02}&0&B^{22}&0\\
                                    0&B^{13}&0&B^{33}
           \ear\right).
\ela
The remaining five equations 
\bela{G67}
  \langle\xsf^1,\asf\rangle=\langle\xsf^5,\asf\rangle=\langle\bsf,\asf\rangle=\langle\xsf^2,\bsf\rangle=\langle\xsf^4,\bsf\rangle = 0
\ela
then simplify considerably and lead to the parametrisation
\bela{G68}
  B^{00} = -\frac{\beta^2}{\beta^0}B^{02},\quad B^{11} = -\frac{\alpha^3}{\alpha^1}B^{13},\quad
 B^{22} = -\frac{\alpha^0}{\alpha^2}B^{02},\quad B^{33} = -\frac{\beta^1}{\beta^3}B^{13},
\ela
%
%
%
%
where the entries $B^{02}$ and $B^{13}$ turn out to be two Pl\"ucker coordinates of the line passing through $\asf$ and $\bsf$, namely
\bela{G69}
  B^{02} = \gamma^{13} ,\quad B^{13} = \gamma^{02} 
\ela
as defined by (\ref{G35}). Accordingly, the following theorem has been retrieved.

\begin{theorem}
For any hexagon in $\C\P^3$ in general position, there exists a unique correlation which maps any line (extended edge) of the hexagon to its opposite line. The correlation is involutive and constitutes a polarity.
\end{theorem}

We now combinatorially attach the lines of the hexagon to six vertices of an elementary cube and the points of intersection $\xsf^i$ to the corresponding edges as depicted in figure \ref{correlation}.
\begin{figure}
\centerline{\includegraphics[scale=0.5]{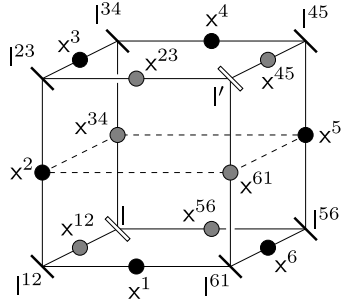}}
\caption{The hexagon with vertices $\xsf^i$ uniquely determines a correlation $\kappa$ which interchanges opposite lines $\lsf^{ii+1}$ and $\lsf^{i+3i+4}$. The correlation $\kappa$ then maps the given line $\lsf$ to a line $\lsf'$. The vertices $\xsf^i$ and the points of intersection $\xsf^{ii+1}$ enjoy the planarity property, that is, for instance, the points $\xsf^2,\xsf^{34},\xsf^5$ and $\xsf^{61}$ are coplanar.}
\label{correlation}
\end{figure}
The line passing through the vertices $\xsf^i$ and $\xsf^{i+1}$ is denoted by $\lsf^{i i+1}$. There exists a one-parameter family of lines $\lsf$ which intersect the lines $\lsf^{12},\lsf^{34}$ and $\lsf^{56}$. Any fixed line $\lsf$ is mapped by the correlation $\kappa$ to a line $\lsf'$ which intersects the lines $\lsf^{23},\lsf^{45}$ and $\lsf^{61}$. Accordingly, the 8 lines $\lsf,\lsf^{12},\lsf^{23},\lsf^{34},\lsf^{45},\lsf^{56},\lsf^{61}$ and $\lsf'$ form an elementary cube of a line complex with the usual property of lines intersecting along edges. On the other hand, as pointed out in the preceding, the hexagon and the line $\lsf$ uniquely determine via the planarity property an eighth line $\hat{\lsf}$ such that the eight lines form an elementary cube of a fundamental line complex. Hence, as indicated above, the connection between the lines $\lsf'$ and $\hat{\lsf}$ is now being examined.

We denote by $\xsf^{ii+1}$ the point of intersection of a line $\lsf$ or its counterpart $\lsf'$ and the line $\lsf^{ii+1}$ as indicated in figure \ref{correlation}. One may construct the one-parameter family of lines $\lsf$ by parametrising the points
\bela{G70}
 \xsf^{12} =  \mu^1\xsf^1 + \mu^2\xsf^2, \quad \xsf^{34} = \mu^3\xsf^3 + \mu^4\xsf^4,\quad\xsf^{56} = \mu^5\xsf^5 + \mu^6\xsf^6
\ela
and solving the collinearity condition
\bela{G71}
  \xsf^{12} + \xsf^{34} + \xsf^{56} = 0
\ela
for $\mu^3,\mu^4,\mu^5,\mu^6$ in terms of $\mu^1$ and $\mu^2$. A brief calculation reveals that
\bela{G72}
 \bear{rlrl}
  \mu^3 = &\dis\frac{\beta^0\mu^1 - \beta^2\mu^2}{\gamma^{02}},&\quad
  \mu^4 = &\dis\frac{\gamma^{03}\mu^1 - \gamma^{23}\mu^2}{\gamma^{02}}\AS
  \mu^6 = &\dis-\frac{\alpha^0\mu^1 - \alpha^2\mu^2}{\gamma^{02}},&\quad
  \mu^5 = &\dis\frac{\gamma^{01}\mu^1 + \gamma^{12}\mu^2}{\gamma^{02}}.
 \ear
\ela
Similarly, the one-parameter family of lines intersecting the lines $\lsf^{23},\lsf^{45}$ and $\lsf^{61}$ may be obtained on use of the parametrisation 
\bela{G73}
 \xsf^{23} =  \nu^2\xsf^2 + \nu^3\xsf^3, \quad \xsf^{45} = \nu^4\xsf^4 + \nu^5\xsf^5,\quad\xsf^{61} = \nu^6\xsf^6 + \nu^1\xsf^1
\ela
of the points of intersection. The collinearity condition
\bela{G74}
  \xsf^{23} + \xsf^{45} + \xsf^{61} = 0
\ela
then determines the coefficients $\nu^6,\nu^1,\nu^2,\nu^3$ in terms of $\nu^4$ and $\nu^5$. The latter two parameters are fixed (up to scaling) by the condition that $\lsf'$ be the image of $\lsf$ under the correlation $\kappa$. This is equivalent to demanding that, for instance,
\bela{G75}
  \langle\xsf^{12},\xsf^{23}\rangle = 0
\ela
and it turns out that $\nu^i = \mu^i$. Finally, one may directly verify that, for instance,
\bela{G76}
  |\xsf^2,\xsf^{34},\xsf^5,\xsf^{61}| = 0
\ela
so that the points $\xsf^2,\xsf^{34},\xsf^5,\xsf^{61}$ are seen to be coplanar as indicated in figure~\ref{correlation}. This implies that, remarkably, the lines $\lsf^{ii+1}$ and $\lsf,\lsf'$ form an elementary cube of a fundamental line complex in $\C\P^3$ and therefore $\lsf'=\hat{\lsf}$. Hence, a characterisation of fundamental line complexes in $\C\P^3$ in terms of correlations has been uncovered.

\begin{theorem}
Fundamental line complexes in $\C\P^3$ are line complexes with the property that neighbouring lines intersect and opposite lines of any elementary cube are interchanged by a correlation which, necessarily, constitutes a polarity. Furthermore, the planarity and concurrency properties associated with any elementary cube of a fundamental line complex are interchanged by the corresponding correlation in the sense that any four coplanar diagonals are mapped to four concurrent diagonals and vice versa. 
\end{theorem}

\section{Conclusions}

The superposition principle (\ref{G10}) for the squared eigenfunctions $M^{ik}$ associated with binary Darboux transformations is standard in the algebraic and geometric theory of continuous and discrete integrable systems. In this paper, we interpret this superposition principle as a stand-alone discrete integrable system, namely the $M$-system (\ref{G3}), and discuss its algebraic and geometric properties. In order to highlight its universality, we have briefly indicated its direct connection with conjugate lattices and the hexahedron recurrence. The main aim of this paper is to introduce a novel correspondence between the $M$-system and fundamental line complexes represented as lattices in the Pl\"ucker quadric. In fact, the Pl\"ucker coordinates of the lines turn out to be the entries $M^{ik}$ and, more generally, the minors of the matrix $\mathcal{M}$. This point of view is custom-made for a detailed analysis of special fundamental line complexes associated with sub-geometries. This is the subject of a forthcoming publication. For instance, it is well known that the binary Darboux transformation associated with the CKP hierarchy of integrable equations requires the squared eigenfunctions to be "symmetric". Indeed, it is easy to see that the constraint $M^{ik}=M^{ki}$ is compatible with the $M$-system. In geometric terms, this means that one considers the intersection of the Pl\"ucker quadric with a hyperplane, resulting in a three-dimensional quadric which one may identify with the Lie quadric of Lie circle geometry. In this manner, the intersecting lines of the fundamental line complexes are contained in a linear complex and may be reinterpreted as touching oriented circles. As a consequence, these circle complexes are governed by the dCKP equation alluded to in Section 3. An analogous approach is also available in the context of Lie sphere geometry. A key ingredient in the geometric treatment of these special fundamental line complexes is the characterisation of fundamental line complexes in terms of the correlations discussed in Section 5.

\end{document}